\documentclass{amsart}

\usepackage[utf8]{inputenc}

\usepackage{amsmath, amsfonts,amssymb,amsthm,array}
\usepackage{graphicx}
\usepackage{stackengine}
\usepackage{thmtools,thm-restate}
\usepackage{enumitem}
\usepackage{float}
\usepackage{comment}
\usepackage{xifthen}
\usepackage{latexsym}
\usepackage{stmaryrd}
\usepackage{enumitem}
\usepackage[bookmarks,hyperfootnotes=false]{hyperref}
\usepackage{multirow}
\usepackage{xcolor}
\usepackage{caption}
\colorlet{darkishRed}{red!80!black}
\colorlet{darkishBlue}{blue!60!black}
\colorlet{darkishGreen}{green!60!black}
\usepackage{hyperref}
\hypersetup{
pdftitle = {Ends of Digraphs},
pdfsubject = {Ends of Digraphs},
pdfauthor = {Carl Bürger and Ruben Melcher},
pdfkeywords = {},
draft = false,
bookmarksopen=true}

\newtheorem{innercustomthm}{Theorem}

\newenvironment{customthm}[1]
  {\renewcommand\theinnercustomthm{#1}\innercustomthm}
  {\endinnercustomthm}

\newtheorem{theorem}{Theorem}[section]
\newtheorem{mainresult}{Theorem}

\newtheorem{proposition}[theorem]{Proposition}%[section]
\newtheorem{cor}[theorem]{Corollary}
\newtheorem{lemma}[theorem]{Lemma}

\theoremstyle{definition}

\setenumerate{label={\normalfont (\roman*)},itemsep=0pt}

\colorlet{darkishGreen}{green!60!black}

\newcommand{\cX}{\mathcal{X}}

\newcommand{\N}{\mathbb{N}}

\newcommand{\lno}{sensitive order}
\newcommand{\nao}[1][]{%
\ifthenelse{\equal{#1}{}}{\trianglelefteq_T}{\trianglelefteq_{T_{#1}}\!}%
}

\newcommand{\dc}[1]{\lceil #1\rceil}
\newcommand{\uc}[1]{\lfloor #1\rfloor}

\newcommand{\dS}{\mathstrut\mkern2.5mu S\mkern-13mu\raise1.4ex%
\hbox{$\leftrightarrow$}}

\def\lowfwd #1#2#3{{\mathop{\kern0pt #1}\limits^{\kern#2pt\raise.#3ex
\vbox to 0pt{\hbox{$\scriptscriptstyle\rightarrow$}\vss}}}}
\def\lowbkwd #1#2#3{{\mathop{\kern0pt #1}\limits^{\kern#2pt\raise.#3ex
\vbox to 0pt{\hbox{$\scriptscriptstyle\leftarrow$}\vss}}}}

\def\ve{\kern-1.5pt\lowfwd e{1.5}2\kern-1pt}
\def\ev{\kern-1pt\lowbkwd e{0.5}2\kern-1pt}
\def\vf{\kern-2pt\lowfwd f{2.5}2\kern-1pt}

\def\Tv{\lowbkwd T{0.3}1}
\def\Dv{\lowbkwd D{0.3}1}

\stackMath

{\begin{list}{}%
          {\setlength{\leftmargin}{#1}}%
          \item[]%
}
{\end{list}}

\graphicspath{{Main/}} %dieser befehl sorgt dafuer das die bilder aus dem Main ordner genommen werden

\begin{document}

\title[Ends of digraphs]{Ends of digraphs III: normal arborescences} 

\author{Carl B\"{u}rger}
\author{Ruben Melcher}
\address{University of Hamburg, Department of Mathematics, Bundesstraße 55 (Geomatikum), 20146 Hamburg, Germany}
\email{carl.buerger@uni-hamburg.de, ruben.melcher@uni-hamburg.de}


\maketitle

\begin{abstract}
In a series of three papers we develop an end space theory for digraphs. Here in the third paper we introduce a concept of depth-first search trees in infinite digraphs, which we call \emph{normal spanning arborescences}. 

We show that normal spanning arborescences are end-faithful: every end of the digraph is represented by exactly one ray in the normal spanning arborescence that starts from the root. 
We further show that this bijection extends to a homeomorphism between the end space of a digraph $D$, which may include limit edges between ends, and the end space of any normal arborescence with limit edges induced from $D$.
Finally we prove a Jung-type criterion for the existence of normal spanning arborescences.
\end{abstract}

\section{Introduction}
\noindent Ends of graphs are one of the most important concepts for the study of infinite graphs.
In a series of three papers we develop an end space theory for digraphs.
See~\cite{EndsOfDigraphsI} for a comprehensive introduction to the entire series of three papers  (\cite{EndsOfDigraphsI}, \cite{EndsOfDigraphsII} and this paper) and a brief overview of all our results.

Depth-first search trees are a standard tool in finite graph and digraph theory. These trees arise from an algorithm on a graph or digraph called \emph{depth-first search}. Starting from a fixed vertex, the `root', the algorithm moves along the edges, going to a vertex not visited yet whenever this is possible, and going back otherwise. Depth-first search stops when all vertices have been visited, and the trees defined by the traversed edges are called \emph{depth-first search} trees. 

For connected finite graphs, the depth-first search trees are precisely the normal spanning trees. Here, a rooted tree $T\subseteq G$  is \emph{normal} in $G$ if the endvertices of every $T$-path in $G$ are comparable in the tree-order of $T$. (A \emph{$T$-path} in $G$ is a non-trivial path that meets $T$ exactly in its endvertices.) Normal spanning trees generalise depth-first search trees, since they are also defined for infinite graphs; they are perhaps the single most important structural tool in infinite graph theory~\cite{DiestelBook5}. 

In this third paper of our series we introduce and study normal spanning arborescences. These are generalisations of depth-first search trees to infinite digraphs that promise to be as powerful for a structural analysis of digraphs as normal spanning trees are for graphs, both from a combinatorial and a topological point of view.

An \emph{arborescence} is a rooted oriented tree $T$ that contains for every vertex \mbox{$v \in V(T)$} a directed path from the root to $v$. 
The vertices of any arborescence are partially ordered as $v \leq_T w$ if $T$ contains a directed path from $v$ to $w$. We write $\uc{v}_T$ for the up-closure of $v$ in $T$.

Consider a finite digraph $D$ together with a spanning depth-first search tree $T\subseteq D$.
If $vw$ is an edge of $D$ between $\le_T$-incomparable vertices of $T$, then $w$ is visited at an earlier stage of the depth-first search than $v$.\footnote{Indeed, if $v$ was visited before $w$, the algorithm would have traversed the edge $vw$ rather than backtracking from $v$, which it must have done since $v$ and $w$ are incomparable. Note that all the visits to $v$ happen while the algorithm searches $\uc{v}_T$, and likewise for $w$, so visiting `before' and `after' are well-defined for incomparable vertices.} Together with all such edges, $T$ forms an acyclic subdigraph of $D$~\cite{cormen2009book}.\footnote{Indeed, any cycle would, but cannot, lie in the up-closure of its first-visited vertex.} 
\begin{figure}[h]
\centering
\def\svgwidth{0.35\columnwidth}
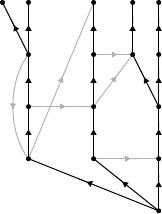
\caption{A depth-first search arborescence visiting vertices from right to left.} 
\label{fig:NSA}
\end{figure} 

Let us use this property of depth-first search trees in finite digraphs as the definition of our infinite analogue, i.e., as the defining property for `normal' arborescences in infinite digraphs. More precisely, consider a (possibly infinite) digraph $D$ and an arborescence $T\subseteq D$, not necessarily spanning. A \emph{$T$-path} in $D$ is a non-trivial directed path that meets $T$ exactly in its endvertices. The \emph{normal assistant} of $T$ in $D$ is the auxiliary digraph $H$ that is obtained from $T$ by adding an edge $vw$ for every two $\le_T$-incomparable vertices $v,w \in V(T)$ for which there is a $T$-path from $\uc{v}_T$ to $\uc{w}_T$ in $D$, regardless of whether $D$ contains such an edge.
The arborescence $T$ is \emph{normal} in $D$ if the normal assistant of $T$ in $D$ is acyclic. It is straightforward to check that this indeed generalises depth-first search trees in that  for finite $D$ a spanning arborescence $T$ of $D$ is normal in $D$ if and only if $T$ defines a depth-first search tree; see Corollary~\ref{corollary: finite dfs trees are NSAs}.

One aspect of why normal spanning trees of infinite undirected graphs are so useful is that they are end-faithful. A spanning tree $T$ of a graph $G$ is \emph{end-faithful} if the map that assigns to every end of $T$ the end of $G$ that contains it as a subset (of rays) is bijective, see \cite{DiestelBook5}. Equivalently $T$ is end-faithful if every end of $G$ is represented by a unique ray in $T$ that starts from a fixed root. Our first theorem will be an analogue of this for normal arborescences, so let us recall the definition of ends of digraphs from~\cite{EndsOfDigraphsI}.

A \emph{directed ray} is an infinite directed path that has a first vertex (but no last vertex). The directed subrays of a directed ray are its \emph{tails}. For the sake of readability we shall omit the word `directed' in  `directed path' and `directed ray' if there is no danger of confusion. We call a ray in a digraph \emph{solid} in $D$ if it has a tail in some strong component of $D-X$ for every finite vertex set $X\subseteq V(D)$. We call two solid rays in a digraph $D$ \emph{equivalent} if for every finite vertex set $X\subseteq V(D)$ they have a tail in the same strong component of $D-X$. The equivalence classes of this equivalence relation are the \emph{ends} of $D$. For a finite vertex set $X\subseteq V(D)$ and an end $\omega$ of $D$ we write $C(X,\omega)$ for the unique strong component of $D-X$ that contains a tail of every ray that represents $\omega$; the end $\omega$ is then said to \emph{live} in that strong component. The set of ends of $D$ is denoted by $\Omega(D)$.

Let $T\subseteq D$ be a spanning arborescence of a digraph $D$. We say that $T$ is \emph{end-faithful} if every end of $D$ is represented by a unique ray in $T$ starting from the root of~$T$. (Note that, conversely, rays in $T$ will only represent ends of $D$ if they are solid in $D$.) Here is our first main result:

\begin{customthm}{\ref{thm: normal arborescences are end-faithful}}
Every normal spanning arborescence of a digraph is end-faithful.
\end{customthm}
\noindent In fact we will prove a localised version of this for  normal arborescences in $D$ that are not  necessarily spanning.

The end space of any normal spanning tree $T$ of an undirected graph $G$ coincides with the end space of $G$, not only combinatorially but also topologically. Indeed, the map that assigns to every end of $T$ the end of $G$ that contains it as a subset is a homeomorphism between the end space of $T$ and that of $G$, see \cite{DiestelBook5}. Hence, in order to understand the end space of $G$ one just needs to understand the simple structure of the tree $T$.

We also have an analogue of this for digraphs and their normal arborescences. To state this, let us recall the notion of limit edges of a digraph $D$.

For two distinct ends $\omega $ and $\eta$ of $D$, we call the pair  $(\omega,\eta)$ a \emph{limit edge} from $\omega$~to~$\eta$ if $D$ has an edge from $C(X,\omega)$ to $C(X,\eta)$ for every finite vertex set $X$  for which $\omega$ and $\eta$ live in distinct strong components of $D-X$. Similarly, for a vertex $v\in V(D)$ and an end $\omega$ of $D$ we call the pair $(v,\omega)$ a \emph{limit edge} \emph{from} $v$ \emph{to} $\omega$ if $D$ has an edge from $v$ to $C(X,\omega)$ for every finite vertex set $X\subseteq V(D)$ with $v \not\in  C(X, \omega)$. And we call the pair $(\omega,v)$ a \emph{limit edge} \emph{from} $\omega$ \emph{to} $v$ if $D$ has an edge from $C(X,\omega)$ to $v$ for every finite vertex set $X\subseteq V(D)$ with $v \not\in C(X, \omega)$. 
The digraph $D$, its ends, and its limit edges together form a topological space $|D|$, in which the edges are copies of the real interval $[0,1]$; see~\cite{EndsOfDigraphsII}.

The \emph{horizon} of a digraph $D$ is the subspace of $|D|$ formed by the ends of $D$ and all the limit edges between them.
Arborescences do not themselves have ends or limit edges, but there is a natural way to endow an arborescence $T$ in a digraph $D$ with a meaningful horizon. The \emph{solidification} of an arborescence $T \subseteq D$, or of its normal assistant $H$ in $D$, is obtained from $T$ or $H$, respectively, by adding all the edges $wv$ with $vw\in E(T)$. Note that all the rays of $T$ are solid in its solidification and thus represent ends there. Let us define the \emph{horizon} of $T$ as the horizon of the solidification of its normal assistant.

Recall that the digraphs $D$ that are compactified by $|D|$ are precisely the \emph{solid} ones, those such that $D-X$ has only finitely many strong components for every finite vertex set $X \subseteq V(D)$~\cite{EndsOfDigraphsII}. Let $T$ be a normal spanning arborescence of $D$, with root $r$, say. By Theorem~\ref{thm: normal arborescences are end-faithful}, there exists a well-defined map~$\psi$ that sends every end $\omega$ of $D$ to the end of the solidification $\overline{T}$ of $T$ represented by the unique ray $R\subseteq T$ starting from $r$ that represents $\omega$ in $D$. This map $\psi$ is clearly injective. If $D$ is solid, every ray in $T$ represents an end of $D$, so $\psi$ is also surjective. Let~$\zeta$ denote the map from the set of ends of $\overline{T}$ to that of the solidification $\overline{H}$ of the normal assistant $H$ of $T$ in $D$ that assigns to every end of $\overline{T}$ the end of $\overline{H}$ that contains it as a subset (of rays).  This is always bijective, see Lemma~\ref{lemma: Zeta is immer bij.}.  Note that $\overline{H}$, unlike $\overline{T}$, can have limit edges. We say that $T$ \emph{reflects the horizon} of $D$ if the map $\zeta \circ \psi \colon \Omega(D)\to \Omega(\overline{H})$ extends to a homeomorphism from the horizon of $D$ to that of $\overline{H}$.

As our second main result we prove that normal spanning arborescences of solid digraphs reflect the horizon of the digraph they span:
\begin{customthm}{\ref{thm: normal arborescence and end space}}
Every normal spanning arborescence of a solid digraph reflects its horizon. 
\end{customthm}
Not every connected graph has a normal spanning tree; for example, uncountable complete graphs have none. Thus it is not surprising that there are also strongly connected digraphs without normal spanning arborescences---such as any digraph obtained from an uncountable complete graph by replacing every edge by its two orientations as separate directed edges.

Jung~\cite{jung69} characterised the connected graphs with a normal spanning tree in terms of dispersed sets. A set $U\subseteq V(G)$ of vertices of a graph $G$ is \emph{dispersed} if there is no comb in $G$ with all its teeth in $U$. Recall that a \emph{comb} is the union of a ray $R$  with infinitely many disjoint finite paths, possibly trivial, that have precisely their first vertex on~$R$. The last vertices of those paths are the \emph{teeth} of this comb, see~\cite{DiestelBook5}. Jung proved that a connected graph has a normal spanning tree if and only if its vertex set is a countable union of dispersed sets.

Translating this to digraphs, a \emph{directed comb} is the union of a directed ray with infinitely many disjoint finite paths (possibly trivial) that have precisely their first vertex on~$R$. Hence the underlying graph of a directed comb is an undirected comb. The \emph{teeth} of a directed comb are the teeth of the underlying comb.
We call a set $U\subseteq V(D)$ of vertices of a digraph $D$ \emph{dispersed} if there is no directed comb in $D$ with all its teeth in $U$. For two vertices $v,w\in V(D)$, we say that $v$ \emph{can reach} $w$ if $D$ contains a path from $v$ to $w$. 
\begin{customthm}{\ref{thm: normal arborescence existence}}
Let $D$ be any digraph and suppose that $r\in V(D)$ can reach all the vertices of $D$. If $V(D)$ is a countable union of dispersed sets, then $D$ has a normal spanning arborescence rooted in~$r$.
\end{customthm}

\noindent In fact we will prove a slightly stronger version of this where we show how to find a normal arborescence in $D$ that contains a given set of vertices of $D$.

In an undirected graph, the levels of any normal spanning tree are dispersed, so the forward implication in Jung's characterisation is easy. Theorem~\ref{thm: normal arborescence existence} implies the harder backward implication when applied to the digraph obtained from the graph by replacing every edge by its two orientations as separate directed edges.

The easy forward implication in Jung's theorem does not have a directed analogue, since the converse implication in Theorem~\ref{thm: normal arborescence existence} may fail (see Section~\ref{section: existence}). However, the converse of Theorem~\ref{thm: normal arborescence existence} does hold if the digraph $D$ is solid.
% For example consider the digraph obtained from a ray by adding a new vertex $r$ and an edge from $r$ to every vertex of the ray. In this digraph, the first level of the normal arborescence that consists of all edges at $r$ is not dispersed.

This paper is organised as follows.  We provide the tools and terminology that we use throughout this paper in Section~\ref{section: tools and terminology}. Then in Section~\ref{section: Normal arborescences} we introduce normal arborescences and provide some basic lemmas that we need for the proofs of our main results. In Section~\ref{section: end-faithful}, we show that normal spanning arborescences are end-faithful, Theorem~\ref{thm: normal arborescences are end-faithful}. In Section~\ref{section: end space}, we prove that normal spanning arborescences reflect the horizon, Theorem~\ref{thm: normal arborescence and end space}. Finally, we prove our existence criterion for normal arborescences in digraphs, Theorem~\ref{thm: normal arborescence existence}, in Section~\ref{section: existence}.

\section{Tools and terminology}\label{section: tools and terminology} 
\noindent Any graph-theoretic notation not explained here can be found in Diestel's textbook~\cite{DiestelBook5}. For the sake of readability, we sometimes omit curly brackets of singletons, i.e., we write $x$ instead of $\{x\}$ for a set~$x$.
Furthermore, we omit the word `directed'---for example in `directed path'---if there is no danger of confusion.

Throughout this paper $D$ is an infinite digraph without multi-edges and without loops, but which may have inversely directed edges between distinct vertices. For a digraph $D$, we write $V(D)$ for the vertex set of $D$, we write $E(D)$ for the edge set of $D$ and $\cX(D)$ for the set of finite vertex sets of $D$. We write edges as ordered pairs $(v,w)$ of vertices $v,w\in V(D)$, and  we usually write $(v,w)$ simply as $vw$.   The vertex $v$ is the \emph{tail} of $vw$ and the vertex $w$ its \emph{head}. The \emph{reverse} of an edge $vw$ is the edge $wv$. More generally, the \emph{reverse} of a digraph $D$ is the digraph on $V(D)$ where we replace every edge of $D$ by its reverse, i.e., the reverse of $D$ has the edge set $\{\, vw \mid wv \in E(D)\,\}$. We write $\Dv$ for the reverse of a digraph $D$.
A \emph{symmetric ray} is a digraph obtained from an undirected ray by replacing each of its edges by its two orientations as separate directed edges. Hence the reverse of a symmetric ray is a symmetric ray.

The directed subrays of a ray are its \emph{tails}. Call a ray \emph{solid} in $D$ if it has a tail in some strong component of $D-X$ for every finite vertex set $X\subseteq V(D)$.
Two solid rays in $D$ are \emph{equivalent}, if they have a tail in the same strong component of $D-X$ for every finite vertex set $X \subseteq V(D)$. We call the equivalence classes of this relation the \emph{ends} of $D$ and we write $\Omega(D)$ for the set of ends of $D$.

Similarly, the reverse subrays of a reverse ray are its \emph{tails}. We call a reverse ray \emph{solid} in $D$ if it has a tail in some strong component of $D-X$ for every finite vertex set $X\subseteq V(D)$. With a slight abuse of notation, we say that a reverse ray $R$ \emph{represents} an end $\omega$ if there is a solid ray $R'$ in $D$ that represents $\omega$ such that $R$ and $R'$ have a tail in the same strong component of $D-X$ for every finite vertex set $X\subseteq V(D)$.

Given sets $A, B\subseteq V(D)$ of vertices a \emph{path from $A$ to $B$}, or \emph{$A$--$B$} path is a path that meets $A$ precisely in its first vertex and $B$ precisely in its last vertex. We say that a vertex $v$ can \emph{reach} a vertex $w$ in $D$ and $w$ can be \emph{reached} from $v$ in $D$ if there is a $v$--$w$ path in $D$. A non-trivial path $P$ is an $A$-path for a set of vertices $A$ if $P$ has both its endvertices but none of its inner vertices  in $A$. A set $W$ of vertices is \emph{strongly connected} in $D$ if every vertex of $W$ can reach every other vertex of $W$ in $D[W]$.

A vertex set $Y \subseteq V(D)$ \emph{separates} $A$ and $B$ in $D$ with $A,B\subseteq V(D)$ if every $A$--$B$ path meets $Y$, or if every $B$--$A$ path meets $Y$. For two vertices $v$ and $w$ of $D$ we say that $Y\subseteq V(D) \setminus \{v,w\}$ \emph{separates} $v$ and $w$ in $D$, if it separates $\{v\}$ and $\{w\}$ in~$D$.

For a finite vertex set $X \subseteq V(D)$ and a strong component $C$ of $D-X$ an end $\omega$ is said to \emph{live in} $C$ if one (equivalent every) solid ray in $D$ that represents $\omega$ has a tail in $C$. We write $C(X,\omega)$ for the strong component of $D-X$ in which $\omega$ lives. For two ends $\omega$ and $\eta$ of $D$ a finite set $X \subseteq V(D)$ is said to \emph{separate} $\omega$ and $\eta$ if $C(X, \omega) \neq C(X, \eta )$, i.e., if $\omega$ and $\eta$ live in distinct strong components  of $D-X$.

We say that a digraph is acyclic if it contains no directed cycle as a subdigraph. The vertices of any acyclic digraph $D$ are  partially ordered by $v\le_D w$ if $D$ contains a path from $v$ to $w$. 

An \emph{arborescence} is a rooted oriented tree that contains for every vertex $v\in V(T)$ a directed path from its root to $v$. Note that arborescences $T$ are acyclic and that $\le_T$ coincides with the tree-order of the undirected tree underlying $T$. For vertices $v\in V(T)$, we write $\uc{v}_T$ for the up-closure and $\dc{v}_T$ for the down-closure of $v$ with regard to $\le_T$. The $n$th \emph{level} of $T$ is the $n$th level of the undirected tree underlying~$T$. 

A \emph{directed comb} is the union of a ray with infinitely many finite disjoint paths (possibly trivial) that have precisely their first vertex on $R$. Hence the undirected graph underlying a directed comb is an undirected comb. The \emph{teeth} of a directed comb are the teeth of the underlying undirected comb. The ray from the definition of a directed comb is the \emph{spine} of the directed comb.

Let $H$ be any fixed digraph. A \emph{subdivision} of $H$ is any digraph that is obtained from $H$ by replacing every edge $vw$ of $H$ by a path $P_{vw}$ with first vertex $v$ and last vertex $w$ so that the paths $P_{vw}$ are  internally disjoint and do not meet~$V(H)\setminus \{v,w\}$. We call the paths $P_{vw}$ \emph{subdividing} paths. If $D$ is a subdivision of $H$, then the original vertices of $H$ are the \emph{branch vertices} of $D$ and the new vertices its \emph{subdividing vertices}. 

\emph{An inflated} $H$ is any digraph that arises from a subdivision $H'$ of $H$ as follows. Replace every branch vertex $v$ of $H'$ by a strongly connected digraph $H_v$ so that the $H_v$ are disjoint and do not meet any subdividing vertex; here replacing means that we first delete $v$ from $H'$ and then add $V(H_v)$ to the vertex set and $E(H_v)$ to the edge set. Then replace every subdividing path $P_{vw}$ that starts in $v$ and ends in $w$ by an $H_v$--$H_w$ path that coincides with $P_{vw}$ on inner vertices. We call the vertex sets $V(H_v)$ the \emph{branch sets} of the inflated $H$. A \emph{necklace} is an inflated symmetric ray with finite branch sets; the branch sets of a necklace are its \emph{beads}. With a slight abuse of notation, we say that a necklace $N\subseteq D$ \emph{represents} an end $\omega$ of $D$ if one (equivalently every) ray in $N$ represents~$\omega$. Given a set $U$ of vertices in a digraph $D$, a necklace $N\subseteq D$ is \emph{attached to} $U$ if infinitely many of the branch sets of $N$ contain a vertex  from $U$.

For two distinct ends $\omega, \eta\in \Omega(D)$, we call the pair $(\omega,\eta)$ a \emph{limit edge} from $\omega$ to~$\eta$, if $D$ has an edge from $C(X,\omega)$ to $C(X,\eta)$ for every finite vertex set $X\subseteq V(D)$ that separates $\omega$ and $\eta$. For a vertex $v\in V(D)$ and an end $\omega\in \Omega(D)$ we call the pair $(v,\omega)$ a \emph{limit edge} \emph{from} $v$ \emph{to} $\omega$ if $D$ has an edge from $v$ to $C(X,\omega)$ for every finite vertex set $X\subseteq V(D)$ with $v \not\in  C(X, \omega)$. Similarly, we call the pair $(\omega,v)$ a \emph{limit edge} \emph{from} $\omega$ \emph{to} $v$ if $D$ has an edge from $C(X,\omega)$ to $v$ for every finite vertex set $X\subseteq V(D)$ with $v \not\in C(X, \omega)$. We write $\Lambda(D)$ for the set of all the limit edges of $D$. 
As we do for `ordinary' edges of a digraph, we will suppress the brackets and the comma in our notation of limit edges. For example we write $\omega\eta$ instead of $(\omega,\eta)$ for a limit edge between ends $\omega$ and $\eta$. For limit edges we need the following proposition from the first paper of this series~\cite[Proposition~5.2]{EndsOfDigraphsI}.

\begin{proposition}\label{proposition: char limit edge type II}
For a digraph $D$, a vertex $v$ and an end $\omega$ of $D$ the following assertions are equivalent:
\begin{enumerate}
    \item $D$ has a limit edge from $v$ to $\omega$;
    \item there is a necklace $N\subseteq D$ that represents $\omega$ such that $v$ sends an edge to every bead of $N$.
\end{enumerate}
\end{proposition}

For vertex sets $A, B\subseteq V(D)$ let $E(A,B)$ be the set of edges from $A$ to $B$, i.e., $E(A,B)= (A\times B)\cap E(D)$. Now, consider two ends $\omega, \eta\in \Omega(D)$ and a finite vertex set $X\subseteq V(D)$. If $X$ separates $\omega$ and $\eta$ we write $E(X,\omega\eta)$ as short for $E(C(X,\omega),C(X,\eta))$ and if additionally $\omega\eta$ is a limit edge, then we say that it \emph{lives} in $E(X,\omega\eta)$.

\section{Normal arborescences}\label{section: Normal arborescences}
\noindent In this section we introduce normal arborescences and we provide some basic lemmas that we need for the proofs of our main results.

Consider a  digraph $D$ and an arborescence $T\subseteq D$, not necessarily spanning.  The \emph{normal assistant} of $T$ in $D$ is the auxiliary digraph $H$ that is obtained from $T$ by adding an edge $vw$ for every two $\le_T$-incomparable vertices $v,w \in V(T)$ for which there is a $T$-path from $\uc{v}_T$ to $\uc{w}_T$ in $D$, regardless of whether $D$ contains such an edge.
The arborescence $T$ is \emph{normal} in $D$ if the normal assistant of $T$ in $D$ is acyclic;
in this case, we write $\nao:=\le_H$ and we call $\nao$ the \emph{normal order} of $T$. 
Similarly, a reverse arborescence $T$ is \emph{normal} in $D$ if $\Tv$ is normal in $\Dv$.

\begin{lemma}\label{lemma: cycle in normal assistant}
Let $D$ be any digraph and let $T\subseteq D$ be an arborescence. If the normal assistant of $T$ in $D$ contains a cycle, then it also contains a cycle so that consecutive vertices on the cycle are $\le_T$-incomparable.
\end{lemma}

\begin{proof}
Let $H$ be the normal assistant of $T$ in $D$ and let $C$ be a cycle in $H$ of minimal length. Suppose for a contradiction that $C$ contains consecutive vertices that are $\le_T$-comparable. As $T$ is acyclic the cycle $C$ cannot be contained entirely in $T$; in particular $C$ has length at least three. Thus we find a subpath $uvw\subseteq C$ such that $u$ is the $\le_T$-predecessor of $v$, and such that $v$ and $w$ are $\le_T$-incomparable.  But then also $uw\in E(H)$ and replacing the path $uvw$ in $C$ by the edge $uw$ gives a shorter cycle. 
\end{proof}

An extension $\preceq$ of $\le_T$ on an arborescence $T$ is \emph{branch sensitive} if for any two $\le_{T}$-incomparable vertices $v \preceq w$ of $T$ there is no $v'\in \uc{v}_T$ with $w \preceq v'$. An extension $\preceq$ of $\le_T$ on $T$ is \emph{path sensitive} if for no two  $\le_T$-incomparable vertices $v\preceq w$ the digraph $D$ contains a $T$-path from $w$ to $v$. Note that the normal order of any normal arborescence $T\subseteq D$ is both branch sensitive and path sensitive. A \emph{\lno } on $T$ is a linear extension of $\le_T$ on $T$ that is both branch sensitive and path~sensitive.

\begin{lemma}\label{lemma: normal order and linear normal order}
Let $D$ be any digraph and let $T\subseteq D$ be an arborescence in $D$. Then $T$ is normal in $D$ if and only if there is a \lno\ on $T$.
\end{lemma}

\begin{proof}
For the forward implication assume that $T$ is normal in $D$. Let us write $L_n$ for the $n$th level of $T$ and let us write $T_n$ for the arborescence that $T$ induces on $\bigcup \{\, L_m \mid m\le n\,\}$.  We recursively construct an ascending sequence of orders $(\preceq_n)_{n\in \N}$ such that $\preceq_n$ is a \lno\ on $T_n$ as follows. In the base case, we let $\preceq_0:= \le_{T_0}$. In the recursive step, suppose that we have defined $\preceq_n$.  Let us write  for every $v\in L_n$ the set of up-neighbours (children) of $v$ in $T$  as $N_v$. For every $v\in L_n$ let  $\preceq_v$ be a linear extension of the restriction of $\nao$ to $N_v$. And for every two distinct vertices $v,w\in L_n$ with $v\preceq_n w$ we define $v'\preceq_{vw}w'$ whenever $v'\in N_v$ and $w'\in  \dc{N_w}_T\setminus \dc{v}_T$. Now, let $\preceq_{n+1}$ be the transitive closure of $$\preceq_n \cup \bigcup \{\, \preceq_v \mid v \in L_n\,\} \cup \bigcup \{\, \preceq_{vw}\mid v\neq w \text{ in } L_n\,\}.$$
It is straightforward to check that the order $\preceq_{n+1}$ is a \lno\ on $T_{n+1}$. Hence \mbox{$\bigcup \{\, \preceq_n \mid n \in \N \, \}$} is a \lno\ on $T$ as an ascending union of \lno s on subarborescences of $T$. 

For the backward implication assume that $T$ has a \lno\ $\preceq$ on~$T$. Suppose for a contradiction that $T$ is not normal in $D$. Let $H$ be the normal assistant of $T$ in $D$. Then $H$ contains a cycle $C$ and by Lemma~\ref{lemma: cycle in normal assistant} we may assume that consecutive vertices on $C$ are $\le_T$-incomparable. Let $c$ be the $\preceq$-largest vertex on $C$ and let $c'$ be its successor on $C$. Note that $c'\preceq c$ by the choice of $c$. The edge $cc'$ of $C\subseteq H$ is witnessed by a $T$-path $P$ from $\uc{c}_T$ to $\uc{c'}_T$. Let $w$ be the first vertex and $v$ the last vertex of $P$. 
As $\preceq$ is branch sensitive, we have $v\preceq w$. But then the two vertices $v$ and $w$ together with $P$ show that $\preceq$ is not path sensitive contradicting that $\preceq$ is a \lno\ on $T$.
\end{proof}

\begin{cor}\label{corollary: finite dfs trees are NSAs}
A spanning arborescence of a finite digraph is normal if and only if it defines a depth-first search tree. 
\end{cor}

\begin{proof} Let $T$ be a spanning arborescence of a finite digraph $D$. 
For the forward implication assume that $T$ is normal in $D$. By Lemma~\ref{lemma: normal order and linear normal order}, we find a \lno\ $\preceq$ on~$T$. Then $T$ is defined by the traversed edges of  the depth-first search  that starts in the root of $T$ and  always chooses the $\preceq$-largest up-neighbour (child) in $T$ in each step.

For the backward implication assume that $T$ is a depth-first search tree and suppose for a contradiction that $T$ is not normal in $D$. Then the normal assistant of $T$ contains a cycle $C$ and by Lemma~\ref{lemma: cycle in normal assistant} we may choose $C$ so that consecutive vertices on $C$ are $\le_T$-incomparable. Let $x$ be the vertex on $C$ that is visited first in the depth-first search and let $y$ be its successor on $C$. 
The edge $xy$ of the normal assistant of $T$ is witnessed by a $T$-path from $\uc{x}_T$ to $\uc{y}_T$. As $T$ is spanning this path is just an edge $e$. Note that all vertices in $\uc{x}_T$ are visited earlier than those  in $\uc{y}_T$ in the depth-first search. Hence the edge $e$ should have been visited by the depth-first search; this is a contraction because $e$ is not an edge of~$T$. 
\end{proof}

We think of (countable) normal spanning arborescence $T\subseteq D$ as being drawn in the plane with all the edges between $\le_T$-incomparable vertices running from left to right; see Figure~\ref{fig:NSA}.

% NICHT LÖSCHEN 

%To make this a little more formal, we recursively define an injective map $\varphi$ from $V(T)$ to the plane, so that for each two $\le_T$-incomparable vertices $v\nao w$ the point  $\varphi(v)$ is  smaller than $\varphi(w)$ in its first entry: By Lemma~\ref{lemma: normal order and linear normal order} we find a \lno\ $\preceq$ on $V(T)$. We begin by mapping the root of $T$ to the point $(1,0)$ of the plane.  Then, we embed every up-neighbour $v$ of the root of $T$ to a point $\varphi(v)=(x_v,1)$ and we choose some $\varepsilon_v>0$ such that the following two conditions are satisfied. The first condition is that for every two distinct up-neighbours $v\preceq w$ of the root of $T$ the point $\varphi(v)$ is  smaller than $\varphi(w)$ in its first entry. The second condition is that no vertex is mapped to any point in $\{\, (x,1) \mid  x_v - \varepsilon_v \le  x\le x_v\, \}$. Next, consider each vertex $v$ of the first level of $T$ separately and embed its up-neighbours in the set $\{\, (x,2) \mid  x_v - \varepsilon_v \le x\le x_v\, \}$ accordingly. We continue this process in infinitely many steps by working our way up the levels of $T$.

Let us see that, similar to their undirected counterparts, normal arborescences capture the separation properties of $D$, while they carry the simple structure of an arborescence:

\begin{lemma}\label{lemma: separation props of NSA}
Let $D$ be any digraph and let $T\subseteq D$ be a normal arborescence in $D$. If $v,w\in V(T)$ are $\le_T$-incomparable vertices of $T$ with $w\ntrianglelefteq_T  v$, then every $w$--$v$ path in $D$ meets $X:=\dc{v}_T\cap \dc{w}_T$. In particular, $X$ separates $v$ and $w$ in $D$. 
\end{lemma}

\begin{proof} Suppose for a contradiction that $P$ is  a $w$--$v$ path in $D$ that avoids $X$, for $\leq_T$-incomparable vertices $v,w\in V(T)$ with $w\ntrianglelefteq_T  v$. Let $N_X$ consist of all neighbours of $X$ in the digraph $T$ that are contained in $V(T-X)$, let $N_X^1$ consist of all vertices $y\in N_X$ with $y\nao v$ and let $N_X^2:=N_X\setminus N_X^1$. Moreover, let $Z_i$ be the union of the up-closures $\uc{s}_T$ with $s\in N_X^i$, for $i=1,2$. Note that $Z_1$ and $Z_2$ partition $V(T-X)$. As $\nao$ is branch sensitive, we  observe that any two vertices $z_1, z_2\in V(T)\setminus X$ with  $z_1\in Z_1$ and $z_2\in Z_2$ are either incomparable with regard to $\nao$, or satisfy $z_1\nao  z_2$. Let $z_1$ be the first vertex of $P$ in $Z_1$ and let $z_2$ be the last vertex of $P$ in $T$ that precedes $z_1$ in the path-order of $P$. Note that $z_2$ is contained in $Z_2$ by our assumption that $P$ avoids $X$. Hence the $T$-path $z_2Pz_1$ witnesses that $z_2 \nao z_1$ contradicting our aforementioned observation. 
\end{proof}

The \emph{dichromatic number} \cite{neumanndichromatic} of a digraph $D$ is the smallest cardinal $\kappa$ so that $D$ admits a vertex partition into $\kappa$ many partition classes that are acyclic in~$D$. From the path sensitivity of normal arborescences we obtain the following:

\begin{proposition}
Every digraph that has a normal spanning arborescence does have a countable dichromatic number. 
\end{proposition}

\begin{proof}
We denote by $L_n$ the $n$th level of $T$ and claim that $L_n$ is acyclic for every $n \in \N$. The vertices in $L_n$ are pairwise $\le_T$-incomparable.  As $T$ is path sensitive there is no $w$-$v$ path in $D[L_n]$ between vertices $v\nao w$ in $L_n$. This would be violated by the $\nao$-largest vertex $w$ and its successor $v$ in $C$ of any directed cycle $C\subseteq D[L_n]$. Hence the non-empty $L_n$ define a partition of $V(D)$ into acyclic vertex sets, witnessing that $D$ has a countable dichromatic number.
\end{proof}

\section{Arborescences are end-faithful}\label{section: end-faithful}
\noindent In this section we prove that normal spanning arborescences capture the end space combinatorially. Let $T\subseteq D$ be a fixed arborescence of a digraph $D$ and let $\Psi$ be a set of ends of $D$. We say that $T$ is \emph{end-faithful for $\Psi$} if every end in $\Psi$ is represented by a unique ray of $T$ that starts from the root. We call the rays in a normal arborescence $T$ that start from the root \emph{normal rays} of $T$. 
We say that an end $\omega$ of $D$ is contained in the  \emph{closure} of a vertex set $U\subseteq V(D)$ if $C(X,\omega)$ meets $U$ for every finite vertex set $X\subseteq V(D)$. Note that an end $\omega$ is contained in the closure of the vertex set of a ray $R$ if and only if $R$ represents $\omega$.

\begin{mainresult}\label{thm: normal arborescences are end-faithful}
Let $D$ be any digraph and let $U\subseteq V(D)$ be any vertex set. If $T$ is a normal arborescence containing $U$, then $T$ is end-faithful for the set of ends in the closure of $U$.
\end{mainresult}

We will employ the following star-comb lemma \cite[Lemma~8.2.2]{DiestelBook5} in order to prove Theorem~\ref{thm: normal arborescences are end-faithful}: 
\begin{lemma}[Star-comb lemma]
Let $W$ be an infinite set of vertices in a connected undirected graph $G$.
Then $G$ contains a comb 
with all its teeth in $W$ or a subdivided infinite star with all its leaves in $W$.
\end{lemma}

\begin{proof}[Proof of Theorem~\ref{thm: normal arborescences are end-faithful}]
First, let $R_1$ and $R_2$ be distinct normal rays of $T$ that represent  ends of $D$ in the closure of $U$, say $\omega_1$ and $\omega_2$, respectively. Our goal is to show that $\omega_1$ and $\omega_2$ are distinct ends of $D$. By Lemma \ref{lemma: separation props of NSA}, the rays $R_1$ and $R_2$ have tails in distinct strong components of $D-X$ for $X=V(R_1)\cap V(R_2)$. Hence $X$ witnesses that $R_1$ and $R_2$ are not equivalent; in particular $\omega_1 \neq \omega_2$.

It remains two show that every end $\omega$ in the closure of $U$ is represented by a normal ray of $T$. We claim that there is a necklace $N$ attached to $U$ in $D$ that represents $\omega$. For this consider the auxiliary digraph $D'$ obtained from $D$ by adding a new vertex $v^*$ and adding new edges $v^*u$, one for every $u\in U$. Since $\omega$ is contained in the closure of $U$, we have that $v^*\omega$ is a limit edge of $D'$. Note, that adding $v^*$ does not change the set of ends, in the sense that every end of $D'$ contains a unique end of $D$ as a subset (of rays), and we may identify the ends of $D'$ with the ends of $D$. Now, Proposition~\ref{proposition: char limit edge type II} yields a necklace $N\subseteq D$ that represents $\omega$ such that~$v^*$ sends an edge to every bead of $N$. By the definition of $D'$, we conclude that $N$ is attached to $U$. 

Having $N$ at hand, fix a vertex from $U$ of every bead of $N$ and let $W$ be the set of these fixed vertices. Now, apply the star-comb lemma in the undirected tree underlying $T$ to  $W$. We claim that the return is a comb. Indeed, suppose for a contradiction that we get a star and let $c$ be its centre. By Lemma~\ref{lemma: separation props of NSA}, the finite set $\dc{c}_T$ separates any two leaves of the star, which is impossible because they are all contained in the necklace $N\subseteq D$.

So the return of the star-comb lemma is indeed a comb and we may assume that its spine $R$, considered as a ray in $T$, is a normal ray. Our aim is to prove that $R$ represents $\omega$ and we may equivalently show that $\omega$ is contained in the closure of $V(R)$. So given a finite vertex set $X\subseteq V(D)$, fix teeth $u$ and $u'$ of the comb that are contained in $C(X,\omega)$. These exist because the teeth of the comb are contained in $W$ and the choice of  $W$. By Lemma~\ref{lemma: separation props of NSA}, the strong component $C(X,\omega)$ contains a vertex of $\dc{u}_T\cap \dc{u'}_T$. As this intersection is included in $R$ we have verified that $C(X,\omega)$ contains a vertex of~$R$. This completes the proof that $\omega$ is contained in the closure of $R$ and with it the proof of this theorem. 
\end{proof}

\begin{cor}
Let $D$ be any digraph and let $U\subseteq V(D)$ be any vertex set. If $T$ is a reverse normal spanning arborescence containing $U$, then $T$ is end-faithful for the set of ends in the closure of $U$.
\end{cor}

\begin{proof}
Applying Theorem~\ref{thm: normal arborescences are end-faithful} to the digraph $\Dv$ and the normal arborescence $\Tv\subseteq \Dv$ shows that $\Tv$ is end-faithful for the ends of $\Dv$ in the closure of $U$. Hence the statement is a consequence of the fact that the ends of $D$ in the closure of $U$ correspond bijectively to the ends of $\Dv$ in the closure of $U$, via the map that sends an end $\omega$ of $D$  to the end of $\Dv$  that is represented by some (equivalently every) reverse ray of $D$ that represents $\omega$.
\end{proof}

\section{Arborescences reflect the horizon}\label{section: end space}

\noindent One of the most useful facts about normal spanning trees is that the end space of any normal spanning tree $T$ coincides with the end space of the graph $G$ it spans---even topologically, i.e., the map that assigns to every end of $T$ the end of $G$ that contains it as a subset is a homeomorphism, see~\cite{DiestelOpenProblems}. Hence, in order to understand the end space of $G$ one just needs to understand the simple structure of the tree $T$.

In~\cite{EndsOfDigraphsII} we defined a topological space $|D|$ formed by a digraph $D$ together with its ends and limit edges.  
The \emph{horizon} of a digraph $D$ is the subspace of $|D|$ formed by the ends of $D$ and all the limit edges between them. 
In order to understand the results of this section it is not necessary to know the topology on $|D|$, as the subspace topology on the horizon of $D$ is particularly simple. Let us give a brief description of the subspace topology for the horizon of $D$.

The ground set of the horizon of a digraph $D$ is defined as follows. Take the set of ends $\Omega(D)$ of $D$ together with a copy $[0,1]_\lambda$ of the unit interval for every limit edge  $\lambda$ between two ends of $D$. Now, identify every  end $\omega$ with the copy of $0$ in $[0,1]_\lambda $ for which $\omega$ is the tail of $\lambda$ and  with the copy of $1$ in $[0,1]_{\lambda'}$ for which $\omega$ is the head of~$\lambda'$, for all the limit edges $\lambda $ and $ \lambda'$ between ends of $D$.
For inner points $z_\lambda \in [0,1]_\lambda$ and $z_{\lambda'} \in  [0,1]_{\lambda'}$ of limit edges $\lambda$ and $\lambda'$ between ends of $D$ we say that $z_\lambda$ \emph{corresponds} to $z_{\lambda'}$ if both correspond to the same point of the unit interval.

We describe the topology of the horizon of $D$ by specifying the basic open sets. Neighbourhoods $\Omega_{\varepsilon}(X,\omega)$ of an end $\omega$ are of the following form: Given $X\in \cX(D)$ let $\Omega_{\varepsilon}(X,\omega)$ be the union of

\begin{itemize}
    \item the set of all the ends that live in $C(X,\omega)$ and the points of limit edges between ends that live in $C(X,\omega)$ and
    \item half-open partial edges  $(\varepsilon,y ]_{\lambda}$ respectively $[ y,\varepsilon)_\lambda$ for every limit edge $\lambda$ between ends for which $y$ lives in $C(X,\omega)$. 
\end{itemize}

Neighbourhoods $\Lambda_{\varepsilon,z}(X,\lambda)$ of inner points $z$ of a limit edge $\lambda$ between  ends  are of the following form: Given $X\in \cX(D)$ that separates the endpoints of $\lambda$ let $\Lambda_{\varepsilon,z}(X,\lambda)$ be the union of all the open balls of radius $\varepsilon$ around points $z_{\lambda'}$ with $\lambda'$ a limit edge between ends that lives in the bundle $E(X,\lambda)$ and with $z_{\lambda'}$ corresponding to $z$. Here we make the convention that for limit edges $\lambda$ between ends the $\varepsilon$ of open balls $B_\varepsilon(z)$ of radius $\varepsilon$ around points $z \in \lambda$ is implicitly chosen small enough to guarantee $B_\varepsilon(z)\subseteq \lambda$.

Arborescences do not themselves have ends or limit edges, but there is a natural way to endow an arborescence $T$ in a digraph $D$ with a meaningful horizon. The \emph{solidification} of an arborescence $T \subseteq D$, or of its normal assistant $H$ in $D$, is obtained from $T$ or $H$, respectively, by adding all the edges $wv$ with $vw\in E(T)$. Note that all the rays of $T$ are solid in its solidification and thus represent ends there. Let us define the \emph{horizon} of $T$ as the horizon of the solidification of its normal assistant.

Now suppose that we have fixed a root $r$ of $T$ and suppose that every ray of $D$ is solid in $D$. By Theorem~\ref{thm: normal arborescences are end-faithful}, there exists a well-defined map~$\psi$ that sends every end $\omega$ of $D$ to the end of the solidification $\overline{T}$ of $T$ represented by the unique ray $R\subseteq T$ starting from $r$ that represents $\omega$ in $D$. This map $\psi$ is clearly injective.  Note that the map $\psi$ is also surjective, by our assumption that every ray of $D$ is solid in $D$. Let $\zeta$ denote the map from the set of ends of $\overline{T}$ to that of the solidification $\overline{H}$ of the normal assistant $H$ of $T$ in $D$ that assigns to every end of $\overline{T}$ the end of $\overline{H}$ that contains it as a subset (of rays). This is always bijective:  

\begin{lemma}\label{lemma: Zeta is immer bij.}
Let $D$ be a digraph, $T$ a normal spanning arborescence of $D$ and $H$ the normal assistant of $T$ in $D$. The map $\zeta \colon \Omega(\overline{T}) \to \Omega( \overline{H})$ that assigns to every end of $\overline{T}$ the end of $\overline{H}$ that contains it as a subset is bijective.
\end{lemma}
\begin{proof}
To see that $\zeta$ is injective, let $\omega_1$ and $\omega_2$ be distinct ends of $\overline{T}$ and let $R_i$ be the ray in $T$ starting from the root of $T$ that represents $\omega_i$ for $i=1,2$.  By Lemma~\ref{lemma: separation props of NSA} the two rays $R_1$ and $R_2$ have a tail in distinct strong components of $\overline{H}-X$ for $X:=\dc{R_1}_T\cap \dc{R_2}_T$; hence $\zeta$ maps the ends $\omega_1$ and $\omega_2$ to distinct ends of~$\overline{H}$.

To see that $\zeta$ is onto, let $\omega$ be an end of $\overline{H}$ and let $R$ be any solid ray in  $\overline{H}$ that represents $\omega$. Our goal is to find a solid ray $R'$ in $\overline{T}$ that is equivalent to $R$ in $\overline{H}$: then the end of $\overline{T}$ that is represented by $R'$ is included in $\omega$ as a subset of rays. For this apply the star-comb lemma in the undirected tree underlying $T$ to the vertex set of $R$. If the return is a comb, then the comb's spine defines the desired ray $R'$.   Indeed, the paths between the comb's spine and its teeth, define (in $\overline{H}$) a family of disjoint directed paths from $R'$ to $R$ and from $R$ to $R'$, hence $R'$ and $R$ are equivalent in~$\overline{H}$. It now suffices to show that the return of the star-comb lemma is always a comb; so suppose for a contradiction that it is  a star with centre $c$ say. Then, by Lemma~\ref{lemma: separation props of NSA}, the down-closure of $c$ in $T$ separates infinitely many vertices of $V(R)$ in $\overline{H}$, contradicting that $R$ has a tail in a strong component of $\overline{H}-X$ for every finite vertex set $X$.
\end{proof}

Note that $\overline{H}$, unlike $\overline{T}$, can have limit edges. We say that $T$ \emph{reflects the horizon} of $D$ if the map $\zeta \circ \psi \colon \Omega(D)\to \Omega(\overline{H})$ extends to a homeomorphism from the horizon of $D$ to that of $\overline{H}$.

The horizon of a normal spanning arborescences might differ from the horizon of the digraph it spans. 
This is due to the nature of connectivity in digraphs: a digraph might have a normal spanning arborescence and many strong components at the same time. For example consider the digraph $D$ depicted in Figure~\ref{fig:CombOfRays}. On the one hand, every ray in $D$ is solid in $D$. On the other hand, consider the unique normal spanning arborescence $T$ that is rooted in the leftmost vertex of the bottom ray. Note that $T$ coincides with its normal assistant. Hence $T$ is normal in $D$ and the end $\omega$ in the horizon of $T$ is a limit point of the ends $\omega_i$ in the horizon of $T$. In contrast to that, all points in the horizon of $D$ are isolated as every end lives in exactly one strong component of~$D$.
\hspace{1cm}
\begin{figure}[h]
\centering
\def\svgwidth{0.9\columnwidth}
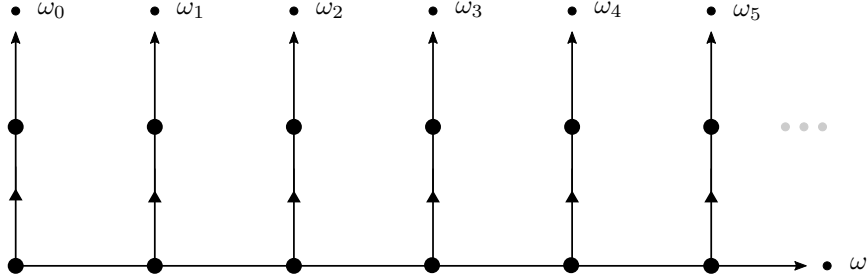
\caption{A digraph $D$ with a normal spanning arborescence $T$ where the horizon of $T$ differs from that of $D$. Every undirected edge in the figure represents a pair of inversely directed edges. Every line that ends with an arrow stands for a symmetric ray.} 
\label{fig:CombOfRays}
\end{figure}

However, it turns out that the horizon of a digraph $D$ coincides with the horizon of any normal spanning arborescence of $D$ if $D$ belongs to an important class of digraphs, namely to the class of solid digraphs: 

\begin{mainresult}\label{thm: normal arborescence and end space}
Every normal spanning arborescence of a solid digraph reflects its horizon. 
\end{mainresult}
The proof of this will be a consequence of the following two lemmas:

\begin{lemma}\label{lemma: horizon map continuous and open}
Let $D$ be a solid digraph, let $T$ be a normal spanning arborescence of $D$ and $H$ the normal assistant of $T$ in $D$. For  ends $\omega$ of $D$ and $\omega'$ of $\overline{H}$, with $\omega' = \zeta(\psi(\omega))$ the following statements hold: 
\begin{enumerate}
    \item For every finite vertex set $X \subseteq V(D)$ there is a finite vertex set $X' \subseteq V(\overline{H})$ such that the vertex set of  $ C( X' , \omega' ) $ is contained in  that of $ C(X, \omega)$.
    \item  For every finite vertex set $X' \subseteq V(\overline{H})$ there is a finite vertex set $X \subseteq V(D)$ such that the vertex set of $ C(X,\omega) $ is contained in that of  $ C(X', \omega')$.
\end{enumerate}
\end{lemma}
\begin{proof}
(i) Let $X$ be any finite vertex set of $D$. We may assume that $X$ is down-closed with respect to $\leq_T$. We write $R_{\omega'}$ for the unique ray of $T$ that starts from the root of $T$ and represents $\omega'$, Theorem~\ref{thm: normal arborescences are end-faithful}. Note that since $D$ is solid every ray in $T$ is solid in~$D$. We will find a vertex $v \in R_{\omega'}$ such that the up-closure of  $v$ in $T$ is contained in $C(X,\omega)$; then $X':= \dc{v}_T \setminus \{v\}$ is as desired by the separation properties of normal arborescences, Lemma~\ref{lemma: separation props of NSA}.

For this, we call a vertex $v\in R_{\omega'}$  \emph{bad} if $\lfloor v \rfloor_T$ meets $V(T)\setminus C(X, \omega)$. Let us show that $R_{\omega'}$ has  only finitely many bad vertices. As $T$ is normal in $D$ and $X$ is down-closed we have that every strong component of $D-X$ other than $C(X,\omega)$ receives at most one edge of $T$ from $C(X, \omega)$. Now, using that $D-X$ has only finitely many strong components, it follows that only finitely many edges of $T$ leave $C(X,\omega)$. Let $B$ be the finite set of all tails of edges of $T$ that leave $C(X,\omega)$. Then also $\dc{B}_T$ is finite and no vertex of $R_{\omega'} - \dc{B}_T$ is bad. This shows that there are indeed only finitely many bad vertices on $R_{\omega'}$. Now, choosing a vertex $v$ on $R_{\omega'}$ higher than any bad vertex and high enough so that the subray of $R_{\omega'}$ that starts at $v$ is included in $C(X,\omega)$ gives $\lfloor v  \rfloor_T  \subseteq C(X,\omega)$. 

(ii) Let $X'$ be any finite vertex set of $\overline{H}$. By the separation properties of normal arborescences, Lemma~\ref{lemma: separation props of NSA}, the finite vertex set $X:= \dc{X'}_T$ is as desired. 
\end{proof}

\begin{lemma}\label{lemma: Zeta erhaelt  limit kanten}
Let $D$ be a solid digraph, $T$ a normal spanning arborescence of $D$ and $H$ the normal assistant of $T$ in $D$. Then $\omega\eta$ is a limit edge of $D$ if and only if $\omega'\eta'$ is a limit edge of $\overline{H}$, where $\omega'$ and $\eta'$ is the image under the map $\zeta \circ \psi$ of $\omega$ and $\eta$, respectively.  
\end{lemma}
\begin{proof}
We write $\omega'$ and $\eta'$ for the image under the map $\zeta \circ \psi$ of $\omega$ and $\eta$, respectively. Let us first show that $\omega' \eta'$ is a limit edge of $\overline{H}$ if $\omega\eta$ is a limit edge of $D$. For this let any finite vertex set $X'$ that separates $\omega'$ and $\eta'$ in $\overline{H}$ be given. Our goal is to find an edge in $\overline{H}$ from $C(X', \omega' )$ to $C(X' , \eta ')$. By Theorem~\ref{thm: normal arborescences are end-faithful} there are rays $R_\omega$ and $R_\eta$ in $T$ that represent $\omega$ and $\eta$ in $D$, respectively. As $\omega \eta$ is a limit edge of $D$, there is an edge in $D$ from $\uc{v_\omega}_T$ to $\uc{v_\eta}_T$ for any two $\leq_T$-incomparable vertices $v_\omega \in R_\omega$ and $v_\eta \in R_\eta$. Now, choose such vertices $v_\omega$ and $v_\eta$  so that both   $\uc{v_\omega}_T$ and $\uc{v_\eta}_T$ avoid $X'$. In $\overline{H}$ both $\uc{v_\omega}_T$ and $\uc{v_\eta}_T$ are strongly connected, by the definition of the solidification. And as $R_\omega$ and $R_\eta$ have a tail in  $C(X', \omega' )$ and $C(X' , \eta')$, respectively, we have that $\uc{v_\omega}_T \subseteq C(X', \omega' )$ and $\uc{v_\eta}_T \subseteq C(X', \eta' )$. In particular, $\uc{v_\omega}_T \cap \uc{v_\eta}_T =\emptyset$. Consequently, any edge  in $D$ from $\uc{v_\omega}_T$ to  $\uc{v_\eta}_T$ has $\leq_T$-incomparable endvertices and therefore is an edge in $\overline{H}$ from  $C(X', \omega' )$ to $C(X' , \eta ')$.

Now, let $\omega' \eta'$ be a limit edge of $\overline{H}$. We write $\omega$ and $\eta$ for the unique preimage under $\zeta \circ \psi$ of $\omega'$ and $\eta'$, respectively.  We  show that $\omega \eta$ is a limit edge in $D$. For this let any finite vertex set $X$ that separates $\omega$ and $\eta$ in $D$ be given. Our goal is to find an edge in $D$ from $C(X, \omega)$ to $C(X , \eta)$. As in the proof of Lemma~\ref{lemma: horizon map continuous and open}, there are vertices $v_\omega$ and $v_\eta$ such that $\uc{v_\omega}_T \subseteq C(X,\omega)$ and $\uc{v_\eta}_T \subseteq C(X,\eta)$. Let~\hbox{$X'=( \dc{v_\omega}_T \cup \dc{v_\eta}_T)\setminus \{v_\omega,v_\eta\}$} and consider $C(X', \omega')$ and $C(X' , \eta')$ in $\overline{H}$. Using that $T$ is normal in $D$ it is easy to show that  $C(X', \omega') =\uc{v_\omega}_T $ and $C(X' , \eta') = \uc{v_\eta}_T$. As $\omega'\eta'$ is a limit edge of $\overline{H}$ there is an edge $e$ in $\overline{H}$ from  $C(X', \omega') $ to $C(X' , \eta')$. Furthermore, the endpoints of $e$ are $\leq_T$-incomparable.  Now, $e$ was added to $T$ in the definition of $H$ because there is an edge $f$ of $D$ from $\uc{v_\omega}_T$ to  $\uc{v_\eta}_T$ and this edge $f$ is as desired.
\end{proof}

\begin{proof}[Proof of Theorem~\ref{thm: normal arborescence and end space}]
By Lemma~\ref{lemma: Zeta is immer bij.} and its preceding text, the map $\zeta \circ \psi$ is a bijection. We extend this map to a bijection $\Theta$ between the horizon of $D$ and that of $\overline{H}$ as follows. Let $y$ be an inner point of a limit edge $\omega\eta$ between ends of~$D$. We write $\omega'$ and $\eta'$ for the image under $\zeta \circ \psi$ of $\omega$ and $\eta$, respectively. By Lemma~\ref{lemma: Zeta erhaelt  limit kanten} we have that $\omega' \eta'$ is a limit edge of $\overline{H}$. Then we declare $\Theta(y):= y'$ for $y'$ the point that corresponds to $y$ on $\omega'\eta'$. Again, by Lemma~\ref{lemma: Zeta erhaelt  limit kanten}, the map $\Theta$ is bijective; we claim that $\Theta$ is even a homeomorphism. Indeed, using Lemma~\ref{lemma: horizon map continuous and open}~(ii) it is straightforward to check that $\Theta$ is continuous and using Lemma~\ref{lemma: horizon map continuous and open}~(i) it is straightforward to check that the inverse of $\Theta$ is continuous. 
\end{proof}
%\green{That the inverse of $\Theta$ is continuous alternatively follows from the fact that it is a continuous one-to-one mapping of a compact space onto  a Hausdorff space, cf. Lemma~2.9 and Theorem~1 in \cite{EndsOfDigraphsII}.}

\section{Existence of arborescences}\label{section: existence} \noindent 
Not every digraph with a vertex that can reach all the other vertices has a normal spanning arborescence, for example any digraph $D$ obtained from an uncountable complete graph by replacing every edge by its two orientations as separate directed edges has none. Indeed, if $T$ is a normal arborescence of $D$, then any two of its vertices must be contained in the same ray starting from the root of $T$. Hence $T$ cannot be spanning.
In this section we give a Jung-type existence criterion for normal spanning arborescence.

For a digraph $D$ we call a  set $U \subseteq V(D)$ of vertices \emph{dispersed} in $D$ if there is no comb in $D$  with all its teeth in $U$.  Our main result of this section reads as follows:

\begin{mainresult}\label{thm: normal arborescence existence}
Let $D$ be any digraph, $U\subseteq V(D)$ and suppose that $r\in V(D)$ can reach all the vertices in $U$. If $U$ is a countable union of dispersed sets, then $D$ has a normal arborescence that contains $U$ and is rooted in~$r$.
\end{mainresult}
\noindent The converse of this is false in general. To see this consider the digraph $D=(\omega_1,E)$ with $E= \{\, (\alpha ,\beta) \mid \alpha <\beta\, \}$ and $U=V(D)$. Here $\omega_1$ denotes the first uncountable ordinal. On the one hand, no infinite subset of $\omega_1$ is dispersed, so $\omega_1$ cannot be written as a countable union of dispersed sets. On the other hand, the spanning arborescence that consists of all the edges with tail $0$ is normal in $D$.  

However, the converse of Theorem~\ref{thm: normal arborescence existence} holds in an important case, namely if the digraph $D$ is solid. Indeed, if $D$ is solid then any arborescence $T\subseteq D$ that is normal in $D$ is locally finite by the separation properties of normal arborescences, Lemma~\ref{lemma: separation props of NSA}. Hence the levels of $T$ are finite; in particular, dispersed.

An analogue of Theorem~\ref{thm: normal arborescence existence} holds for reverse normal arborescences:
\begin{cor}
Let $D$ be any digraph and suppose that $U\subseteq V(D)$ is a countable union of dispersed sets in $\Dv$. If $r\in V(D)$ can be reached by all the vertices in $U$, then $D$ has a reverse normal arborescence that contains $U$ and is rooted in $r$.
\end{cor}

\begin{proof}
Apply Theorem~\ref{thm: normal arborescence existence} to the reverse of $D$.
\end{proof}

\begin{proof}[Proof of Theorem~\ref{thm: normal arborescence existence}] Suppose that the vertex set $U$ can be written as a countable union $\bigcup \{ \,U_n\mid n\in\N\,\}$ of sets that are dispersed in $D$. Then we can write $U$ as a collection $\{\, u_\alpha \mid \alpha <\kappa\}$ for a finite or limit ordinal $\kappa$ such that every proper initial segment of the collection is dispersed in $D$ as follows: We may assume that the $U_n$ are pairwise disjoint. Choose a well-ordering ${\le_n}$ of every $U_n$. Then write $u\le u'$ for vertices $u\in U_m$ and $u'\in U_n$ with $m<n$, or with $m=n$ and $u\le_m u'$. 
It is straightforward to show that $\le$ defines a well-ordering of $U$ that is as desired. 

We may assume that for every limit ordinal $\alpha<\kappa$ the vertex $u_\alpha$ coincides with some $u_\xi$ with $\xi<\alpha$; indeed, just increment the subscripts of the $u_\alpha$ by one for $\alpha$ an infinite ordinal, and recursively redefine $u_\alpha$ to be some $u_\xi$ with $\xi<\alpha$ for $\alpha$ a limit ordinal. 

Now, we recursively define ascending sequences $(T_\alpha)_{\alpha<\kappa}$ and $(\preceq_\alpha)_{\alpha<\kappa}$ such that $T_\alpha$ is an arborescence and $\preceq_\alpha$ is a \lno\ of $T_\alpha$ that satisfies the following conditions: 
\begin{enumerate}
    \item $T_\alpha$ contains  $\{\,u_{\xi}\mid \xi\le\alpha\,\}$ cofinally\footnote{A subset $B$ of a poset $A= (A,\le)$ is \emph{cofinal} in  $A$ if for every $a \in A$ there is a $b \in B$ with $a \leq b$.}  with regard to $\leq_{T_\alpha}$;
    \item if $v,w\in T_\alpha$  with $v\preceq_\alpha w$ are distinct and have a common $\le_{T_\alpha}$-predecessor, then $w\in T_\xi$ and $v\notin T_\xi$ for some $\xi<\alpha$; 
    \item there is no infinite strictly ascending sequence of vertices in $T_\alpha$ with regard to  $\preceq_\alpha$.
\end{enumerate}
Once the $T_\alpha$ are defined the arborescence $T:=\bigcup\{\, T_\alpha\mid \alpha<\kappa\,\}$ is as desired; indeed, $\bigcup \{\, \preceq_\alpha \mid \alpha<\kappa \}$ is a \lno\ on $T$ and thus $T$ is normal in $D$ by Lemma~\ref{lemma: normal order and linear normal order}. Finally, $V(T)$ contains $U$ by condition (i).  

Conditions (ii) and (iii) become relevant in the construction of the $T_\alpha$, which now follows. If $\alpha=0$, then let $T_0$ be any $r$--$u_0$ path in $D$ and let $\preceq_0:=\le_{T_0}$. Otherwise $\beta>0$. If $\beta$ is a limit ordinal, then let $T_{\beta}:= \bigcup \{\, T_\alpha \mid \alpha<\beta\,\}$ and $\preceq_\beta=\bigcup \{\, \preceq_\alpha \mid \alpha<\beta \}$. Then $\preceq_\beta$ is a \lno\ on $T_\beta$ as each $\preceq_\alpha$ with $\alpha<\beta$ is a \lno\ on $T_\alpha$.  
Condition (i) for $\beta$ follows from (i) for $\alpha<\beta$ and our assumption that $u_\beta$ coincides with $u_\alpha$ for some $\alpha<\beta$. Similarly, condition (ii) for $\beta$ follows from (ii) for $\alpha<\beta$. Condition (iii) can be seen as follows. Suppose for a contraction that there is an infinite strictly ascending sequence $(w_n)_{n\in \N}$ in $T_\beta$ with regard to $\preceq_\beta$. Apply the star-comb lemma to the set $\{\,w_n\mid n\in \N\,\}$ in the undirected tree underlying $T_{\beta}$. The return is an infinite subdivided undirected star since an undirected comb would give rise to a directed comb in $T_{\beta}$ with all its teeth in $U$; here we use that by (i) every tooth has a vertex of $U$ in its up-closure and that every proper initial segment of $\{\, u_\alpha \mid \alpha < \kappa \} $ is dispersed. Let $Z$ be the set of $\le_{T_{\beta}}$-up-neighbours of the centre of the subdivided star that contain a tooth in their $\le_{T_{\beta}}$-up-closure. Since $\preceq_\beta$ is branch sensitive we may write $Z$ as a strictly ascending collection $Z=\{\,z_n\mid n\in \N\,\}$ with regard to $\preceq_\beta$. Choose $z^*\in Z\cap V(T_\alpha)$ so that $\alpha$ is minimal with $Z\cap V(T_\alpha) \neq \emptyset$. By (ii) we have that $z_n\preceq_\beta z^*$ for every $z_n\neq z^*$ contradicting that the $z_n$ form a strictly ascending sequence with regard to $\preceq_\beta$.

Now, suppose that $\beta =\alpha+1$ is a successor ordinal.  If $T_{\alpha}$ already contains $u_{\beta}$, we let $T_{\beta}:= T_{\alpha}$. Otherwise $u_{\beta}$ is not contained in $T_{\alpha}$. As $r$ can reach $u_{\beta}$ there is a $T_{\alpha}$--$u_{\beta}$ path $P$. By (iii) for $\alpha$, we may choose $P$ such that its first vertex $v_P$ is $\preceq_{\beta}$-maximal among all the starting vertices of $T_\alpha$--$u_\beta$ paths. We let $T_{\beta}:=T_{\alpha}\cup P$. Note that this ensures condition (i) for $T_\beta$.

In order to define $\preceq_\beta$ we only need to describe how the vertices from $\mathring{v}_P P$ relate to the vertices in $T_{\alpha}$. 
We define vertices of $ P-v_P$ to be smaller than all the vertices larger than $v_P$ and larger than all others (with regard to the normal order of $T_{\alpha}$). 
Note that this ensures (ii). Condition (iii) holds because there is no infinite strictly ascending sequence of vertices in $T_\alpha$ with regard to $\preceq_\alpha$ and $T_\beta$ extends $T_\alpha$ finitely. 

It remains to show that $\preceq_\beta$ is a \lno\ on $T_\beta$. That $\preceq_\beta$ is branch sensitive is immediate from the construction so let us prove that it is path sensitive. Suppose for a contradiction that $Q$ is a $T_{\beta}$-path from $w$ to $v$ with $\le_{T_{\beta}}$-incomparable vertices $v\preceq_{\beta} w$. Since $T_{\alpha}$ is normal either $v$ or $w$ are contained in $P-v_P$. If $w\in P-v_P$, then $v_P P w Qv $ is a path violating that $\preceq_\alpha$ is path sensitive; unless $v_P$ and $v$ are $\le_{T_\beta}$ comparable, but then we would have $w\preceq_{\beta} v$ by the definition of $\preceq_\beta$. In the other case, the path $wQ vPu_{\beta}$ would have been a better choice for~$P$. 
\end{proof}

\bibliographystyle{amsplain}
\bibliography{bibliography.bib}
\end{document}